\documentclass[a4paper]{amsart}
\usepackage{amssymb}
 \usepackage[pagewise]{lineno}%\linenumbers
\usepackage{hyperref}
\usepackage{amsmath,amscd}
\usepackage{stmaryrd}
\usepackage[all]{xy}
\usepackage{epsf}
\usepackage{enumerate}
\usepackage{comment}
\usepackage{combelow}
\providecommand{\U}[1]{\protect\rule{.1in}{.1in}}

\newtheorem*{theorem*}{Theorem}
\newtheorem*{theorem1}{Theorem 1}
\newtheorem*{theorem2}{Theorem 2}
\newtheorem*{theorem3}{Theorem 3}
\newtheorem*{example1}{Example 1}
\newtheorem*{example2}{Example 2}
\newtheorem*{example3}{Example 3}

\newtheorem{corollary}{Corollary}
\newtheorem{definition}{Definition}
\newtheorem{example}{Example}
\newtheorem{lemma}{Lemma}

\newtheorem{remark}{Remark}

\newcommand{\pr}{\operatorname{pr}}

\newcommand{\codim}{\operatorname{codim}}
\newcommand{\id}{\operatorname{id}}

\newcommand{\diffto}{\xrightarrow{\raisebox{-0.2 em}[0pt][0pt]{\smash{\ensuremath{\sim}}}}}
\newcommand{\rmap}{\longrightarrow}
\newcommand{\lmap}{\longleftarrow}

\newcommand{\Co}{N^*}
\newcommand{\N}{N}

\newcommand{\8}{\infty}

\begin{document}
\title{The normal form theorem around Poisson transversals}
\author{Pedro Frejlich}
\address{Departamento de Matem\'{a}tica Puc Rio de Janeiro, Rua Marqu\^{e}s de S\~{a}o Vicente 225, G\'{a}vea, Rio de Janeiro RJ 22451-900, Brazil}
\email{frejlich.math@gmail.com}
\author{Ioan M\u{a}rcu\cb{t}}
\address{IMAPP, Radboud University Nijmegen, 6500 GL, Nijmegen, The Netherlands}
\email{i.marcut@math.ru.nl}
\begin{abstract}
We prove a normal form theorem for Poisson structures around Poisson transversals (also called cosymplectic
submanifolds), which simultaneously generalizes Weinstein's symplectic neighborhood theorem from symplectic geometry \cite{Wein71} and Weinstein's splitting theorem \cite{Wein83}. Our approach turns out to be essentially canonical, and as a byproduct, we obtain an equivariant version of the latter theorem.%\\
%\noindent MSC2010: primary 53D17, 53D05; secondary 70H45.
\end{abstract}
\dedicatory{Dedicated to Alan Weinstein on the occasion of his $70$th birthday}

\maketitle

\tableofcontents

\section{Introduction}

This paper is devoted to the study of semi-local properties of {\em Poisson transversals}. These are submanifolds $X$ of a Poisson manifold $(M,\pi)$ that meet each symplectic leaf of $\pi$ {\em transversally} and {\em symplectically}. A
Poisson transversal $X$ carries a canonical Poisson structure, whose leaves are the intersections of leaves of $\pi$ with $X$, and are endowed with the pullback symplectic structure.

Even though this class of submanifolds has very rarely been dealt with in full generality -- much to our dismay and
surprise -- Poisson transversals permeate the whole theory of Poisson manifolds, often playing a quite fundamental role. This lack of specific attention is especially intriguing since they are a special case of several distinguished classes of submanifolds which have aroused interest lately: Poisson transversals are Lie-Dirac submanifolds \cite{Xu03},
Poisson-Dirac submanifolds \cite{CrFer04}, and also Pre-Poisson submanifolds \cite{CattZamb09} (see also
\cite{Zamb11} for a survey on submanifolds in Poisson geometry).

No wonder, then, that Poisson transversals showed up already in the earliest infancy of Poisson geometry, namely, in
the foundational paper of A. Weinstein \cite{Wein83}. Namely, if $L$ is a symplectic leaf and $x\in L$, then a
submanifold $X$ that intersects $L$ transversally at $x$ and has complementary dimension is a Poisson transversal,
and its induced Poisson structure governs much of the geometry transverse to $L$. In fact, a small enough tubular neighborhood of $L$ in $M$ will have the property that all its fibres are Poisson transversals. Such fibrations are nowadays called \emph{Poisson fibrations}, and were studied by Y. Vorobjev in \cite{Vor} -- mostly in connection to the local structure around symplectic leaves -- and also by R. Fernandes and O. Brahic in \cite{RuiOli}. That Poisson fibrations are related to H\ae fliger's formalism of geometric structures described by groupoid-valued cocycles (see \cite{Haef58} and also \cite{Gromov86}) -- of which the ``automatic transversality'' of Lemma \ref{lem : Poisson pulls back Poisson transversals} is also reminiscent -- should not escape notice. %Example \ref{normal form : Transversals to Foliations} below, simple as it may be, attempts to illustrate this.
In fact, in Physics literature, Poisson fibrations have been known for long in the guise of \emph{second class constraints}, and motivated the introduction by P. Dirac of what we know today as the induced Dirac bracket \cite{Dirac50}, which in our language is the induced Poisson structure on the fibres.

The role played by Poisson transversals in Poisson geometry is similar to that played by symplectic submanifolds in
symplectic geometry and by transverse submanifolds in foliation theory (see the examples in the next section). The key
observation is that the transverse geometry around a Poisson transversal $X$ is of non-singular and
\emph{contravariant} nature: it behaves more like a two-form than as a bivector in the directions conormal to $X$.
This allows us to make particularly effective use of the tools of ``contravariant geometry''. In the core of our arguments lies the fact that the contravariant exponential map $\exp_{\mathcal{X}}$ associated to a Poisson spray $\mathcal{X}$ gives rise to a tubular neighborhood adapted to $X \subset (M,\pi)$, in complete analogy with the classical construction of a tubular neighborhood of a submanifold
$X$ in a Riemannian manifold $(M,g)$, thus effectively reducing many problems to the symplectic case.

\medskip

The main result of this paper is a local normal form theorem around Poisson transversals, which simultaneously
generalizes Weinstein's splitting theorem \cite{Wein83} and Weinstein's symplectic neighborhood theorem
\cite{Wein71}. At a Poisson transversal $X$ of $(M,\pi)$, the restriction of the Poisson bivector $\pi|_{X}\in
\Gamma(\bigwedge^2TM|_{X})$ determines two objects:
\begin{itemize}
\item a Poisson structure on $X$, denoted $\pi_{X}$,
\item a nondegenerate two-form on the conormal bundle $p:\Co X\to X$, denoted
\[w_X\in\Gamma(\bigwedge^2NX).\]
\end{itemize}
Let $\widetilde{\sigma}$ be a closed 2-form on $\Co X$, that extends $\sigma:=-w_X$, i.e.\
\[\widetilde{\sigma}|_{T(N^{\ast}X)\vert_{X}}=\sigma.\]
To such an extension we associate a Poisson structure $\pi(\widetilde{\sigma})$ on an open
$U(\widetilde{\sigma})\subset \Co X$ around $X$. The symplectic leaves of $\pi(\widetilde{\sigma})$ are in one-to-one
correspondence with the leaves of $\pi_{X}$; namely if $(L,\omega_L)$ a leaf of $\pi_{X}$, the corresponding leaf of
$\pi(\widetilde{\sigma})$ is an open $\widetilde{L}\subset p^{-1}(L)$ around $L$ endowed with the 2-form
$\omega_{\widetilde{L}}:=p^*(\omega_L)+\widetilde{\sigma}|_{\widetilde{L}}$. The Poisson manifold
$(U(\widetilde{\sigma}),\pi(\widetilde{\sigma}))$ is the local model of $\pi$ around $X$. We will provide a more conceptual description of the local model using Dirac geometry.

\begin{theorem1}\label{theorem 1}
Let $(M,\pi)$ be a Poisson manifold and $X\subset M$ be an embedded Poisson transversal. An open neighborhood of
$X$ in $(M,\pi)$ is Poisson diffeomorphic to an open neighborhood of $X$ in the local model
$(U(\widetilde{\sigma}),\pi(\widetilde{\sigma}))$.
\end{theorem1}

Under stronger assumptions (which always hold around points in $X$) we can provide an even more explicit description
of the normal form. Assuming \emph{symplectic} triviality of the conormal bundle to $X$, the theorem implies a
generalized version of the Weinstein splitting theorem, expressing the Poisson as a product, i.e., in the form
(\ref{superlocal normal form : coordinates}) below. This coincides with Weinstein's setting when we look at (small)
Poisson transversals of complementary dimension to a symplectic leaf.

The proof of Theorem 1 relies on the symplectic realization constructed in \cite{CrMar11} with the aid of global
Poisson geometry, and on elementary Dirac-geometric techniques; the former is the crucial ingredient that allows us to
have a good grasp of directions conormal to the Poisson transversal, and the latter furnishes the appropriate language
to deal with objects which have mixed covariant-contravariant behavior. As an illustration of the strength and
canonicity of our methods, we present as an application the proof of an equivariant version of Weinstein's splitting
theorem. Other applications of the normal form theorem, which reveal the Poisson-topological aspects of Poisson
transversals, will be treated elsewhere.

\begin{theorem2}
Let $(M,\pi)$ be a Poisson manifold and let $G$ be a compact Lie group acting by Poisson diffeomorphisms on $M$. If
$x\in M$ is a fixed point of $G$, then there are coordinates
$(p_1,\ldots,p_n,q_1,\ldots,q_n,y_1,\ldots,y_m)\in\mathbb{R}^{2n+m}$ centered at $x$ such that
\begin{equation}\label{superlocal normal form : coordinates}
\pi=\sum_{i=1}^n\frac{\partial}{\partial q_i}\wedge\frac{\partial}{\partial p_i}+\frac{1}{2}\sum_{j,k=1}^m \varpi_{j,k}(y)\frac{\partial}{\partial y_j}\wedge\frac{\partial}{\partial y_k},
\end{equation}
and in these coordinates $G$ acts linearly and keeps the subspaces $\mathbb{R}^{2n}\times \{0\}$ and $\{0\}\times\mathbb{R}^m$ invariant.
\end{theorem2}

This answers in the negative a question posed by E. Miranda and N. Zung about the necessity of the ``tameness''
condition they assume in their proof of this result in \cite{MirZung06}. We wish to thank E. Miranda for bringing this
problem to our attention.

\medskip

We should probably also say a few words about terminology. Poisson transversals were also referred to as
\emph{cosymplectic submanifolds} in the literature, and this is motivated by the fact that the conormal directions to
such a submanifold are symplectic i.e.\ the Poisson tensor is nondegenerate on the conormal bundle to the submanifold.
Even though this nomenclature is perfectly reasonable, there are several reasons why we decided not to use this name.
Foremost among these:
\begin{enumerate}
\item There is already a widely used notion of a cosymplectic manifold, defined as a manifold of dimension $2n+1$, endowed with a closed 1-form $\theta$ and a closed 2-form $\omega$ such that $\theta\wedge \omega^{n}$ is a volume form.
\item The general point of view of transverse geometric structures is of great insight into
Poisson transversals when we rephrase the problem in terms of Dirac structures and contravariant geometry.
Moreover, the proximity between the dual pairs used in the proof of the normal form theorem, and the gadget of Morita equivalence, which is known to govern the transverse geometry to the symplectic leaves, is too obvious to ignore.
\end{enumerate}

\medskip

\noindent \textbf{Acknowledgments.} We would like to thank Marius Crainic for useful discussions. The first author was supported by the NWO Vrije Competitie project ``Flexibility and Rigidity of Geometric Structures'' no.\ 612.001.101 and the second by the ERC Starting Grant no.\ 279729.

\section{Some basic properties of Poisson transversals}

Let $(M,\pi)$ be a Poisson manifold. A {\bf Poisson transversal} in $M$ is an embedded submanifold ${X} \subset M$
that meets each symplectic leaf of $\pi$ {\em transversally} and {\em symplectically}. We translate both these
conditions algebraically. Let $x\in X$ and let $(L,\omega)$ be the symplectic leaf through $x$. Transversality
translates to
\[T_xX+T_xL=T_xM.\]
Taking annihilators in this equation, we obtain that $\Co_xX\cap \ker(\pi^{\sharp}_x)=\{0\}$, or equivalently, that the restriction of
$\pi^{\sharp}$ to $\Co_xX$ is injective:
\begin{equation}\label{EQ1}
0\longrightarrow \Co_xX\stackrel{\pi_x^{\sharp}}{\longrightarrow}T_xM.
\end{equation}
For the second condition, note that the kernel of $\omega_x|_{T_xX\cap T_xL}$ is $T_xX\cap \pi_x^{\sharp}(\Co_xX)$. So the
condition that $T_xX\cap T_xL$ be a symplectic subspace is equivalent to
\begin{equation}\label{EQ2}
T_xX\cap \pi_{x}^{\sharp}(\Co_xX)=\{0\}.
\end{equation}
Since $T_xX$ and $\Co_xX$ have complementary dimensions, (\ref{EQ1}) and (\ref{EQ2}) imply the following
decomposition, which is equivalent to $X$ being a Poisson transversal:
\begin{equation}\label{EQ3}
TX\oplus \pi^{\sharp}(\Co X)=TM|_{X}.
\end{equation}

The decomposition of the tangent bundle (\ref{EQ3}) gives canonically an embedded normal bundle, denoted
\[\N{X}:=\pi^{\sharp}(\Co X)\subset TM|_{X},\]
and a corresponding decomposition for the conormal bundle
\begin{equation*}\label{EQ4}
\Co X\oplus \N^{\circ}{X}=T^*M|_{X}.
\end{equation*}
For $\xi\in\Co_xX$ and $\eta\in \N^{\circ}_x{X}$, we have that $\pi^{\sharp}(\xi)\in \N_x X$, hence $\pi(\xi,\eta)=0$. This
implies that $\pi|_{X}$ has no mixed component in the decomposition
\[\bigwedge^2TM|_{X}=\bigwedge^2T{X}\oplus \left(TX\otimes NX\right)\oplus\bigwedge^2NX.\]
Therefore $\pi|_{X}$ splits as
\[\pi|_{X}=\pi_{X}+w_{X},\ \ \pi_{X}\in \Gamma(\bigwedge^2TX), \ w_X\in \Gamma(\bigwedge^2NX).\]

It is well known that these two tensors satisfy the following properties, but for completeness we include a proof.
\begin{lemma}
The bivector $\pi_{X}$ is Poisson and $w_{X}$, regarded as a 2-form on $\Co X$, is fibrewise nondegenerate.
\end{lemma}
\begin{proof}
To prove that $\pi_{X}$ is Poisson, we will use Dirac-geometric techniques (for other approaches, see
\cite{CrFer04,Xu03}; for the basics of Dirac geometry, see \cite{BR}). It suffices to show that the pullback via the
inclusion $i:X\to M$ of the Dirac structure $L_{\pi}:=\{\pi^{\sharp}(\xi)+\xi : \xi\in T^*M\}$ equals the almost Dirac structure
$L_{\pi_{X}}:=\{\pi_{X}^{\sharp}(\xi)+\xi : \xi\in T^*X\}$, since this makes $L_{\pi_{X}}$ automatically involutive, and hence
$\pi_{X}$ Poisson. But to show this it suffices to prove the following inclusion:
\begin{align*}
L_{\pi_{X}}&=\lbrace \pi_{X}^{\sharp}(\xi)+\xi:\xi \in T^{\ast}X \rbrace = \lbrace \pi_{X}^{\sharp}(i^{\ast}\eta)+i^{\ast}\eta:\eta \in N^{\circ}X \rbrace =\\
&=\lbrace \pi^{\sharp}(\eta)+i^{\ast}\eta:\eta \in N^{\circ}X \rbrace  \subset i^{\ast}L_{\pi},
\end{align*}
where we used that $w_{X}^{\sharp}(\eta)=0$, for $\eta\in N^{\circ}X$.

The map $w_{X}^{\sharp}:\Co X\to NX$ is just the restriction of $\pi$, which, by the decomposition (\ref{EQ3}), is a linear isomorphism.
\end{proof}

We recall three natural instances of Poisson transversals, which appear throughout Poisson geometry:

\begin{example1}\rm
If $\pi$ is nondegenerate then $X$ is a Poisson transversal if and only if $X$ is a symplectic submanifold of $(M,\pi)$.
\end{example1}
\begin{example2}\rm
If $L$ is the symplectic leaf of $(M,\pi)$ through a point $x\in M$, a submanifold $X$ that intersects $L$ transversally at $x$ and is of complementary dimension is a Poisson transversal around $x$.
\end{example2}
\begin{example3}\rm
If $(M,\pi)$ is a regular Poisson manifold with underlying foliation $\mathcal{F}$ of codimension $q$, then every submanifold $X$ of dimension $q$ that is transverse to $\mathcal{F}$ is a Poisson transversal.
\end{example3}

A very useful -- and somewhat surprising -- fact about Poisson transversals is that they behave well with respect to Poisson maps:

\begin{lemma}\label{lem : Poisson pulls back Poisson transversals}
Let $\varphi:(M_0,\pi_0)\to (M_1,\pi_1)$ be a Poisson map and $X_1\subset M_1$ be a Poisson transversal. Then:
\begin{enumerate}
 \item $\varphi$ is transverse to $X_1$;
 \item $X_0:=\varphi^{-1}(X_1)$ is also a Poisson transversal;
 \item $\varphi$ restricts to a Poisson map $\varphi|_{X_0}:(X_0,\pi_{X_0}) \to (X_1,\pi_{X_1})$;
 \item The differential of $\varphi$ along $X_0$ restricts to a fibrewise linear isomorphism between embedded normal bundles $\varphi_*|_{NX_0}:NX_0\to NX_1$;
 \item The map $F:N^*X_0\to N^*X_1$, $F(\xi)=(\varphi^*)^{-1}(\xi)$, $\xi\in N^*X_0$ is a fibrewise linear symplectomorphism between the symplectic vector bundles
 \[F:(N^*X_0,w_{X_0})\to (N^*X_1,w_{X_1}).\]
\end{enumerate}
\end{lemma}

\begin{corollary}
Let $(M,\pi)$ be a Poisson manifold, $X\subset M$ be a Poisson transversal and $W\subset M$ be a Poisson
submanifold. Then $W$ and $X$ intersect transversally, and $X \cap W$ is:
\begin{itemize}
 \item a Poisson transversal in $(W,\pi|_{W})$, and
 \item a Poisson submanifold of $(X,\pi_X)$.
\end{itemize}
\end{corollary}

\begin{proof}[Proof of Lemma \ref{lem : Poisson pulls back Poisson transversals}]
Consider $x\in X_0$ and let $y:=\varphi(x)\in X_1$. Since $\varphi$ is a Poisson map we have:
\[\pi_{1}(\eta)=\varphi_{\ast}\left(\pi_0(\varphi^{\ast}\eta)\right),\textrm{ for all }\eta\in T^{\ast}_yM_1,\]
therefore $\pi_{1}(T_y^{\ast}M_1)\subset \varphi_{\ast}(T_xM_0)$. But $X_1$ being a Poisson transversal now implies that $\varphi$ is transverse to
$X_1$:
\[T_yM_1=T_y X_1+\pi_1(T_y^{\ast}M_1)=T_y X_1+\varphi_{\ast}(T_xM_0).\]

In particular, $X_0$ is a submanifold of $M_0$. To show that $X_0$ is a Poisson transversal, we will prove that the
decomposition $T X_0 \oplus \pi_0(N^{\ast} X_0 )=TM_0|_{ X_0 }$ holds. Note first that
\[T_x X_0=(\varphi_{\ast})^{-1}(T_y X_1)\textrm{  and  }N^{\ast}_x X_0=\varphi^{\ast}(N^{\ast}_y X_1)\]
Let $v\in T_xM_0$, and decompose $\varphi_{\ast}v=u+\pi_1(\eta)$, with $u\in T_y X_1$ and $\eta\in N^{\ast}_y X_1$. Then
$\varphi^{\ast}\eta\in N^{\ast}_x X_0$ and $w:=v-\pi_0(\varphi^{\ast}\eta)$ projects to $u$, hence $w\in T_x X_0$. This shows that $v=w+\pi_0(\varphi^{\ast}\eta)\in T_x X_0+\pi_0(N^{\ast}_x X_0)$, hence
\[T_xM_0=T_x X_0+\pi_0(N^{\ast}_x X_0).\]
Counting dimensions, we conclude that this is a direct sum decomposition, and therefore $X_0$ is a Poisson transversal.

Note moreover that $\varphi_*$ preserves the embedded normal bundles:
\[\varphi_{\ast}(N_x X_0)=\varphi_{\ast}(\pi_0(N^{\ast}_x X_0))=\varphi_{\ast}(\pi_0(\varphi^{\ast}(N^{\ast}_y X_1)))=\pi_1(N^{\ast}_y X_1)=N_y X_1,\]
and because they have the same rank, $\varphi_{\ast}|_{NX_0}$ is a fibrewise isomorphism. Since we also have $\varphi_*(T_xX_0)\subset T_yX_1$, the Poisson condition $\varphi_*(\pi_{0,x})=\pi_{1,y}$, implies that $\varphi_*(\pi_{X_{0},x})=\pi_{X_1,y}$ and $\varphi_*(w_{X_0,x})=w_{X_1,y}$. This implies (3) and (4).
\end{proof}

\section{The local model}

The local model around a Poisson transversal depends on an extra choice:

\begin{definition}\rm
Let $(E,\sigma)$ be a symplectic vector bundle over $X$. A {\bf closed extension} of $\sigma$ is a closed 2-form
$\widetilde{\sigma}$ defined on a neighborhood of $X$ in $E$, such that its restriction to $TE|_{X}=TX\oplus E$
equals $\sigma$. We denote the space of all closed extensions by $\Upsilon(E,\sigma)$.
\end{definition}

Closed extensions always exist, and can be constructed employing the standard de Rham homotopy operator (see e.g.
the ``Extension Theorem'' in \cite{Wein77}).

%\begin{lemma}\label{extending closed two-forms}
%Let $i:V\to M$ be an embedded submanifold and let $\beta$ be a $k$-form on $TM|_{V}$ such that $i^*(\beta)\in
%\Omega^k(V)$ is closed. Then $\beta$ extends to a closed $k$-form $\widetilde{\beta}$ on some tubular neighborhood
%of $V$ in $M$.
%\end{lemma}

In the warm-up for the construction below of the local model, let us revisit the three instances which are generalized by our main result.

\begin{example}\rm[Weinstein's Symplectic Neighborhood Theorem, \cite{Wein71}]\label{normal form : weinstein neighborhood}
Let $(M,\omega)$ be a symplectic manifold, and $(X,\omega_{X})\subset M$ be a symplectic submanifold. The
symplectic orthogonal of $TX$, denoted by $E:=TX^{\omega}$, is a symplectic vector bundle with bilinear form
$\sigma:=\omega|_{E}$. The local model around $X$ is given by the closed 2-form
$\widetilde{\sigma}+p^*(\omega_{X})$ on $E$, where $p:E\to X$ is the projection and $\widetilde{\sigma}\in
\Upsilon(E,\sigma)$. Weinstein's symplectic neighborhood theorem says that a neighborhood of $X$ in $(M,\omega)$ is
symplectomorphic to a neighborhood of $X$ in $(E,\widetilde{\sigma}+p^*(\omega_{X}))$.
\end{example}

\begin{example}\rm[Weinstein's Splitting Theorem, \cite{Wein83}]\label{normal form : weinstein spiltting}
Let $(M,\pi)$ be a Poisson manifold and let $x\in M$. Let also $(L,\omega)$ be the symplectic leaf through $x\in M$,
and $(X,\pi_X)$ a Poisson transversal at $x$, of complementary dimension. The local model around $x$ is given
by the product of Poisson manifolds
\begin{equation*}
(T_xL,\omega_{x}^{-1})\times (X,\pi_{X}).
\end{equation*}
Weinstein's Splitting Theorem (or Darboux-Weinstein Theorem) asserts that $(M,\pi)$ is Poisson diffeomorphic around $x$ to an open around $(0,x)$ in the local model.
\end{example}

\begin{example}\rm[Transversals to Foliations]\label{normal form : Transversals to Foliations}
Let $M$ be a manifold carrying a smooth (regular) foliation $\mathcal{F}$, and let $X \subset M$ be a submanifold
transverse to $\mathcal{F}$,
\[T_xX+T_x\mathcal{F}=T_xM, \text{ for all }x \in X.\]
Let $\mathcal{F}_X$ be the induced foliation on $X$. The local model of the foliation $\mathcal{F}$ around $X$ is $(NX,p^{\ast}\mathcal{F}_X)$, where
$p:NX \to X$ is the normal bundle to $X$; note that the leaves of the local model are of the form $p^{-1}(L)$, for
$L$ a leaf of $\mathcal{F}_X$. To build an isomorphism between $\mathcal{F}$ and its model around $X$,
consider a metric $g$ on $T\mathcal{F}$ and let $\exp_{g}:T\mathcal{F} \supset U \to M$
denote the leafwise exponential map of $g$, i.e.\ for each leaf $L$, $\exp_{g}:(TL\cap U) \to L$ is the
(Levi-Civita) exponential map of the Riemannian manifold $(L,g|_L)$. Then $T\mathcal{F}_X^{\perp}\subset T\mathcal{F}|_X$ is a complement to $TX$ in $TM|_X$, and the composition
\[NX\diffto T\mathcal{F}_X^{\perp}\stackrel{\exp_g}{\rmap}M\]
pulls the foliation $\mathcal{F}$ to the local model.
\end{example}

The idea for constructing the local model around a Poisson transversal is to put the foliation in normal form in the
sense of Example \ref{normal form : Transversals to Foliations}, and then perform Weinstein's construction of Example
\ref{normal form : weinstein neighborhood} along all symplectic leaves simultaneously.

Let $(E,\sigma)$ be a symplectic vector bundle over a Poisson manifold $(X,\pi_{X})$ with projection $p:E\to X$ and
consider a closed extension $\widetilde{\sigma}\in \Upsilon(E,\sigma)$. As mentioned in the Introduction, the
symplectic leaves of the local model are $(\widetilde{L},\omega_{\widetilde{L}})$, for $(L,\omega_L)$ a symplectic
leaf of $(X,\pi_{X})$, where $\widetilde{L}\subset p^{-1}(L)$ is an open containing $L$ and
\[\omega_{\widetilde{L}}:=\widetilde{\sigma}|_{\widetilde{L}}+p^*(\omega_{L}).\]
To show that this construction yields a smooth Poisson bivector around $X$, we rewrite it using the language of Dirac
geometry. Let $L_{\pi_{X}}$ be the Dirac structure corresponding to $\pi_{X}$. Dirac structures can be pulled back
along submersions. The pullback of $L_{\pi_{X}}$ to $E$, denoted by $p^*(L_{\pi_{X}})$, has presymplectic leaves
$(p^{-1}(L),p^*(\omega_L))$, where $(L,\omega_L)$ is a symplectic leaf of $\pi_{X}$. Finally, the gauge transform by
$\widetilde{\sigma}$, denoted by $p^*(L_{\pi_{X}})^{\widetilde{\sigma}}$, has the required effect: it adds to each
leaf the restriction of $\widetilde{\sigma}$.

\begin{lemma}\label{Model structure}
Let $(E,\sigma)$ be a symplectic vector bundle over a Poisson manifold $(X,\pi_{X})$, and let $\widetilde{\sigma}\in
\Upsilon(E,\sigma)$ be a closed extension of $\sigma$. On a neighborhood $U(\widetilde{\sigma})$ of $X$ in $E$, we
have that the Dirac structure
\[L(\widetilde{\sigma}):=p^*(L_{\pi_{X}})^{\widetilde{\sigma}}\]
corresponds to a Poisson structure $\pi(\widetilde{\sigma})$ that decomposes along $X$ as
\[\pi(\widetilde{\sigma})|_{X}=\pi_{X}+\sigma^{-1}\in \Gamma(\bigwedge^2TX)\oplus\Gamma(\bigwedge^2E).\]
Equivalently, $(X,\pi_{X})$ is a Poisson transversal for $\pi(\widetilde{\sigma})$, the canonical normal bundle is
$E\subset TE|_{X}$, and the induced nondegenerate bivector is $w_X=\sigma^{-1}$.
\end{lemma}

\begin{proof}
The condition that $L(\widetilde{\sigma})$ be Poisson is open, thus it suffices to show that $L(\widetilde{\sigma})$ has the
expected form along $X$. This can be easily checked, since
\[p^*(L_{\pi_{X}})|_{X}=\{\pi_{X}^{\sharp}(\xi)+Y+\xi : \xi\in T^*X, \ Y\in E\},\]
and therefore
\begin{align*}
L(\widetilde{\sigma})|_{X}&=\{\pi_{X}^{\sharp}(\xi)+Y+\xi+\iota_{Y}\sigma : \xi\in T^*X, \ Y\in E\}=\\
&=\{\pi_{X}^{\sharp}(\xi)+(\sigma^{-1})^{\sharp}(\eta)+\xi+\eta : \xi\in T^*X, \ \eta\in E^*\}=\\
&=\{(\pi_{X}+\sigma^{-1})^{\sharp}(\theta)+\theta : \theta\in T^*E|_{X}\}.\qedhere
\end{align*}
\end{proof}

\begin{definition}\rm
The Poisson manifold $(U(\widetilde{\sigma}),\pi(\widetilde{\sigma}))$ from the Lemma is called the \textbf{local
model} associated to $(E,\sigma)$ and $(X,\pi_{X})$.

If $X$ is a Poisson transversal of a Poisson manifold $(M,\pi)$, $\pi_{X}$ is the induced Poisson structure on $X$,
$E=\Co X$ is the conormal bundle to $X$ and $\sigma=-w_X=-(\pi|_{N^*X})$, then
$(U(\widetilde{\sigma}),\pi(\widetilde{\sigma}))$ is called the \textbf{local model of $\pi$ around} $X$.
\end{definition}

\begin{remark}\rm
We point out that there is a choice in having the local models of $\pi$ around $X$ live in the \emph{conormal} bundle
to $X$, as opposed to its normal bundle $NX$, as is typically the case for normal form theorems. In fact,
since\[w_X:(\Co X,-w_X) \to (NX,w_X^{-1})\]is an isomorphism of symplectic vector bundles, we can translate
canonically all our constructions to $NX$ via $w_X$.

That we chose $\Co X$ instead of $NX$ is meant to emphasize that we regard the conormal $\Co X$ as the
more appropriate notion of ``contravariant normal'', an opinion which is corroborated by the scheme of proof of
Theorem, where we spread out a tubular neighborhood of $X$ by following contravariant geodesics starting in
directions conormal to $X$.
\end{remark}

The construction of the local model depends on the choice of a closed extension. We prove a Poisson version of the Moser
argument, which we later employ to prove that different extensions induce isomorphic local models.

\begin{lemma}[Moser Lemma]\label{Moser Lemma : Poisson}

Suppose we are given a path of Poisson structures of the form $t \mapsto \pi_t:=\pi^{td\alpha}$, where $\pi$ is a
Poisson structure and $\alpha \in \Omega^1(M)$. Then the isotopy $\phi^{t,s}_{\mathcal{V}}$ generated by the
time-dependent vector field $\mathcal{V}_t:=-\pi_t^{\sharp}(\alpha)$ stabilizes $\pi_t$:
\[\phi^{t,s}_{\mathcal{V}\ast}\pi_s=\pi_t,\]
whenever this is defined.
\end{lemma}
\begin{proof}
Recall that Poisson cohomology is computed by the complex $\left(\mathfrak{X}^{\bullet}(M), d_{\pi}\right)$, where
$d_{\pi}:\mathfrak{X}^{\bullet}(M) \to \mathfrak{X}^{\bullet+1}(M)$ is defined by $d_{\pi}:=[\pi,\cdot]$ and
$[\cdot,\cdot]$ stands for the Schouten bracket on multi-vector fields. Moreover, $\pi$, regarded as a map $\pi^{\sharp}:T^{\ast}M \to TM$, induces a chain map
\[(-1)^{\bullet+1}\wedge^{\bullet}\pi^{\sharp}: \left(\Omega^{\bullet}(M), d\right) \rmap \left(\mathfrak{X}^{\bullet}(M), d_{\pi}\right),\]
from the de Rham complex of differential forms, see e.g. \cite{DZ}. In particular,
\[L_{\mathcal{V}_t}\pi_t=[\pi_t,\pi_t^{\sharp}(\alpha)]= d_{\pi_t}\pi_t^{\sharp}(\alpha)=-\wedge^2\pi_t^{\sharp}(d\alpha).\]
As maps, this can be written as $L_{\mathcal{V}_t}\pi_t^{\sharp}=\pi_t^{\sharp}\circ (d\alpha)_{\flat} \circ \pi_t^{\sharp}$. Also, by the very definition of gauge transformation, we have the identity $\pi^{\sharp}=\pi_t^{\sharp}\circ(\id+t(d\alpha)_{\flat} \circ \pi^{\sharp})$, whence
\[\frac{d \pi_t^{\sharp}}{dt}\circ(\id+t(d\alpha)_{\flat}\circ \pi^{\sharp})+\pi_t^{\sharp}\circ (d\alpha)_{\flat}\circ \pi^{\sharp}=0 \Longrightarrow \frac{d\pi_t^{\sharp}}{dt}=-\pi_t^{\sharp}\circ (d\alpha)_{\flat}\circ \pi_t^{\sharp}.\]
Finally, we obtain
\begin{align*}
\frac{d}{dt}(\phi^{t,s}_{\mathcal{V}})^{\ast}\pi_t=(\phi^{t,s}_{\mathcal{V}})^{\ast}\left(L_{\mathcal{V}_t}\pi_t+\frac{d\pi_t}{dt} \right)=0,
\end{align*}
showing that $(\phi^{t,s}_{\mathcal{V}})^{\ast}\pi_t=\pi_s$.
\end{proof}

Next, we show that different choices of closed extensions yield isomorphic local models.

\begin{lemma}\label{lemma_indep}
Let $(E,\sigma)$ be a symplectic vector bundle over a Poisson manifold $(X,\pi_{X})$. All corresponding local models
are isomorphic around $X$ by diffeomorphisms that fix $X$ up to first order.
\end{lemma}
\begin{proof}
If $\widetilde{\sigma}_1\in\Upsilon(E,\sigma)$ is a second extension, $\widetilde{\sigma}_1-\widetilde{\sigma}$ is a closed 2-form on $E$ that vanishes on $TE|_{X}$. Since the inclusion
$X\subset E$ is a homotopy equivalence, $\widetilde{\sigma}_1-\widetilde{\sigma}$ is exact, and one can choose a
primitive $\eta\in \Omega^1(E)$ that vanishes on $TE|_{X}$. Actually, c.f.\ the Relative Poincar\'e Lemma in
\cite{Wein77}, one may choose $\eta$ with vanishing first derivatives along $X$. Denote by
$\pi_0:=\pi(\widetilde{\sigma})$ and by $\pi_1:=\pi(\widetilde{\sigma}+d\eta)$. Then $\pi_1$ is the gauge transform by
$d\eta$ of $\pi_0$, denoted $\pi_1=\pi_0^{d\eta}$. These bivectors can be interpolated by the family of Poisson
structures
\[\pi_t:=\pi_0^{td\eta}, \  t\in [0,1].\]
Now, $\pi_t$ corresponds to the smooth family of Dirac structures
$L_t:=p^*(L_{\pi_{X}})^{\widetilde{\sigma}+td\eta}$, and the set $U\subset \mathbb{R}\times E$ of those points $(t,x)$ where $L_{t,x}$ is Poisson is open. Since $[0,1]\times X\subset U$, there is an open
neighborhood $V$ of $X$ in $E$ such that $[0,1]\times V\subset U$. Thus, $\pi_t$ is defined on $V$ for all $t\in [0,1]$.
By the Moser Lemma \ref{Moser Lemma : Poisson}, we see that the flow of the time-dependent vector field
\[Y_t:=-\pi_t^{\sharp}(\eta)\]
trivializes the family, i.e.\ $(\phi_{Y}^{t,s})^*(\pi_t)=\pi_s$ whenever it is defined. Since $\eta$ and its first
derivatives vanish along $X$, it follows that $\phi_{Y}^{t,s}$ fixes $X$ and that its differential is the identity on
$TE|_{X}$. Arguing as before, the set where $\phi_{Y}^{t,0}$ is defined up to $t=1$ contains an open neighborhood
$V'\subset V$ of $X$, so we obtain a Poisson diffeomorphism \[\phi_Y^{1,0}:(V',\pi_0)\diffto
(\phi_Y^{1,0}(V'),\pi_1).\qedhere\]
\end{proof}

\section{The normal form theorem}

The Normal Form Theorem \ref{theorem 1} for a Poisson structure $(M,\pi)$ around a Poisson transversal $X$ states
that $\pi$ and its local model (built out of $\pi|_{X}$) are isomorphic around $X$. In the symplectic case, this
follows from the Moser argument in a straightforward manner. For general Poisson manifolds, the proof is more
involved. The main difficulty is to put the foliation in normal form; namely, to find a tubular neighborhood of $X$ along
the leaves of $\pi$. If the foliation is regular, such a construction can be performed by restricting a metric to the leaves
and taking leafwise the Riemannian exponential (cf. Example \ref{normal form : Transversals to Foliations}). If $\pi$ is
not regular, it is not a priori clear if these maps glue to a smooth tubular neighborhood of $X$ in $M$. We will use
instead a ``contravariant'' version of this argument in which we replace the classical exponential from Riemannian
geometry by its Poisson-geometric analog: the contravariant exponential. The more surprising outcome is that a
contravariant exponential not only puts the foliation in normal form, but also provides a closed extension \emph{and}
the required isomorphism to the local model. A funny consequence is that a choice of Poisson spray $\mathcal{X}$ for
$(M,\pi)$ puts \emph{all} of its Poisson transversals in normal form canonically and simultaneously !

We start by recalling some notions and results from contravariant geometry:

\begin{definition}\rm
A {\bf Poisson spray} $\mathcal{X} \in \mathfrak{X}^1(T^{*}M)$ on a Poisson manifold $(M,\pi)$ is a vector field on
$T^{*}M$ such that:
\begin{enumerate}
\item $p_{*}\mathcal{X}(\xi)=\pi^{\sharp}(\xi)$, for all $\xi\in T^*M$
\item $m_t^{*}\mathcal{X}=t\mathcal{{X}}$, for all $t>0$,
\end{enumerate}
where $p:T^*M\to M$ is the projection and $m_t:T^{*}M \to T^{*}M$ is the multiplication by $t$. The flow
$\phi_{\mathcal{X}}^t$ of  $\mathcal{X}$ is called the {\bf geodesic flow}.

The {\bf contravariant exponential} of $\mathcal{X}$ is the map
\[\exp_{\mathcal{X}}:U\rmap M, \ \ \ \ \ \xi\mapsto p\circ\phi_{\mathcal{X}}^1(\xi),\]
on an open $U\subset T^*M$ where the geodesic flow is defined up to time 1. By abuse of notation, we will write
$\exp_{\mathcal{X}}:T^*M\to M$, as if it were defined on $T^*M$.
\end{definition}

Poisson sprays exist on every Poisson manifold. For example, if $\nabla$ is a connection on $T^*M$, then the map that
associates to $\xi\in T^*M$ the horizontal lift of $\pi^{\sharp}(\xi)$ is a Poisson spray.

The main feature of Poisson sprays is that they produce symplectic realizations:

\begin{theorem3}\cite{CrMar11}\label{existence of symplectic realizations}
Given $(M,\pi)$ a Poisson manifold and $\mathcal{X}$ a Poisson spray, there exists an open neighborhood  $\Sigma\subset T^*M$ of the zero section, on which the average of the canonical symplectic structure $\omega_{\text{can}} \in
\Omega^2(T^*M)$ under the geodesic flow
\begin{equation}\label{EQ_average}
 \Omega_{\mathcal{X}}:=\int_0^1\left(\phi_{\mathcal{X}}^t\right)^{*}\omega_{\text{can}}dt,
\end{equation}
is a symplectic structure on $\Sigma$, and the projection $p:(\Sigma,\Omega_{\mathcal{X}})\to (M,\pi)$ is a symplectic
realization (i.e.\ a surjective Poisson submersion).
\end{theorem3}

Let $X\subset (M,\pi)$ be a Poisson transversal. As before, we denote by $\pi_{X}$ the induced Poisson structure on
$X$, and by $w_X:=\pi|_{\Co X}$. We are ready to state the main result of this paper.

\begin{theorem1}[Detailed version]\label{theorem normal form, detailed}
Let $(M,\pi)$ be a Poisson manifold and let $X\subset M$ be a Poisson transversal. A Poisson spray $\mathcal{X}$
induces a closed extension of $\sigma:=-w_X$ in a neighborhood of $X$ in $\Co X$, given by
\[\widetilde{\sigma}_{\mathcal{X}}:=-\Omega_{\mathcal{X}} \vert_{\Co X}\in \Upsilon(\Co X,\sigma).\]
The corresponding local model $\pi(\widetilde{\sigma}_{\mathcal{X}})$ is isomorphic to $\pi$ around $X$. Explicitly,
a Poisson diffeomorphism between opens around $X$ is given by the map
\[\exp_{\mathcal{X}}:(\Co X, \pi(\widetilde{\sigma}_{\mathcal{X}}))\diffto (M,\pi).\]
\end{theorem1}

For the proof of Theorem 1, we need some properties of dual pairs. Recall  \cite{Wein71}:

\begin{definition}\rm
A {\bf dual pair} consists of a symplectic manifold $(\Sigma,\Omega)$, two Poisson manifolds $(M_0,\pi_0)$ and
$(M_1,\pi_1)$, and two Poisson submersions
%\[  \xymatrix{
%  & (\Sigma,\Omega) \ar[dl]_s \ar[dr]^t & \\
%  (M_0,\pi_0) & & (M_1,\pi_1)
%  }
% \]
\[(M_0,\pi_0)\stackrel{s}{\lmap}(\Sigma,\Omega)\stackrel{t}{\rmap}(M_1,\pi_1)\]
with symplectically orthogonal fibres:
\[\ker ds^{\Omega}=\ker dt.\]
The pair is called a {\bf full dual pair}, if $s$ and $t$ are surjective.
\end{definition}

Dual pairs and Poisson transversals interact pretty well, as shows the following:

\begin{lemma}\label{Lemma_restricting_dual_pairs}
Let $(M_0,\pi_0)\stackrel{s}{\lmap}(\Sigma,\Omega)\stackrel{t}{\rmap}(M_1,\pi_1)$ be a dual pair, and let
$X_0\subset M_0$ and $X_1\subset M_1$ be Poisson transversals. Then $\overline{\Sigma}:=s^{-1}(X_0)\cap
t^{-1}(X_1)$ is a symplectic submanifold that fits into the dual pair
\[(X_0,\pi_{X_0})\stackrel{s}{\lmap}(\overline{\Sigma},\Omega|_{\overline{\Sigma}})\stackrel{t}{\rmap}(X_1,\pi_{X_1}).\]
\end{lemma}
\begin{proof}
First note that $\overline{\Sigma}$ is the inverse image of the Poisson transversal $X_0\times X_1$ under the Poisson map
\[(s,t):(\Sigma,\Omega)\rmap (M_0,\pi_{0})\times(M_1,\pi_1).\]
By Lemma \ref{lem : Poisson pulls back Poisson transversals}, $(s,t)$ is
transverse to $X_0\times X_1$, $\overline{\Sigma}$ is a symplectic manifold and $(s,t)$ restricts to a Poisson map
\[(s,t):(\overline{\Sigma},\Omega|_{\overline{\Sigma}})\rmap (X_0,\pi_{X_0})\times(X_1,\pi_{X_1}).\]
It remains to show that the maps
\[\overline{s}:=s|_{\overline{\Sigma}}:\overline{\Sigma}\rmap X_0 \quad\textrm{and}\quad \overline{t}:=t|_{\overline{\Sigma}}:\overline{\Sigma}\rmap X_1\]
are submersions with symplectically orthogonal fibres. Let $m_i:=\dim(M_i)$ and $x_i:=\dim(X_i)$. The fact that the
$s$ and $t$ submersions with orthogonal fibres, implies that $\dim(\Sigma)=m_0+m_1$. By transversality of $(s,t)$
and $X_0\times X_1$, we have that $\codim(\overline{\Sigma})=\codim(X_0\times X_1)$; thus
$\dim(\overline{\Sigma})=x_0+x_1$. Now, for a point $p\in \overline{\Sigma}$, one clearly has $\ker d_p\overline{t}\subset
(\ker d_p\overline{s})^{\Omega|_{\overline{\Sigma}}}$; since $\overline{\Sigma}$ is symplectic, it follows that
\[\dim(\ker d_p\overline{s})+\dim(\ker d_p\overline{t})\leq \dim(\overline{\Sigma})=x_0+x_1.\]
On the other hand, we have that $\dim(\ker d_p\overline{s})\geq \dim(\overline{\Sigma})-\dim(X_0)=x_1$, and similarly
$\dim(\ker d_p\overline{t})\geq x_0$. Therefore, we obtain $\dim(\ker d_p\overline{s})=x_1$ and $\dim(\ker
d_p\overline{t})= x_0$. This implies that $d_p\overline{s}$ and $d_p\overline{t}$ are surjective, and that $\ker
d_p\overline{s}$ and $\ker d_p\overline{t}$ are symplectically orthogonal.
\end{proof}

The next Lemma shows how $\pi_0,\pi_1$ and $\Omega$ are related:

\begin{lemma}\label{lemma_Dirac}
 Let
%  \[
%  \xymatrix{
%  & (\Sigma,\Omega) \ar[dl]^s \ar[dr]_t & \\
%  (M_0,\pi_0) & & (M_1,\pi_1)
%  }
% \]
$(M_0,\pi_0)\stackrel{s}{\lmap}(\Sigma,\Omega)\stackrel{t}{\rmap}(M_1,\pi_1)$ be a dual pair. Then the Dirac
structures $L_{\pi_i}$ corresponding to $\pi_i$ satisfy the following relation:
\[  s^{*}(L_{\pi_0})^{-\Omega}=t^{*}(L_{-\pi_1}). \]
\end{lemma}
\begin{proof}
An element $\chi\in s^{*}(L_{\pi_0})^{-\Omega}$ is of the form
\[\chi=Y+s^{*}\xi-\iota_{Y}\Omega, \text{ where}\quad \xi \in T^{*}M_0, \quad s_{*}Y=\pi_0^{\sharp}(\xi).\]
Since also $s_{*}\Omega^{-1}(s^{*}\xi)=\pi_0^{\sharp}(\xi)$, we have that
\[ Y-(\Omega^{-1})^{\sharp}(s^{*}\xi) \in \ker ds = (\Omega^{-1})^{\sharp}(t^{*}T^{*}M_1).\]
Hence there is $\eta \in T^{*}M_1$ such that
\[ Y=(\Omega^{-1})^{\sharp}(s^{*}\xi)-(\Omega^{-1})^{\sharp}(t^{*}\eta).\]
Applying $t_{*}$, respectively $\Omega$, to both sides we find that
\[ t_{*}Y=-t_{*}(\Omega^{-1})^{\sharp}(t^{*}\eta)=-\pi_1^{\sharp}(\eta), \quad \textrm{respectively}\quad s^*\xi-\iota_{Y}\Omega=t^*\eta.\]
Hence
\[\chi=Y+s^{*}\xi-\iota_{Y}\Omega =Y+t^{*}\eta\in t^{*}(L_{-\pi_1}).\]
This shows one inclusion; the other follows by symmetry.
\end{proof}

As a first step in the proof of Theorem \ref{theorem 1}, we analyze what happens infinitesimally.

\begin{lemma}\label{lemma_well_defined}
We have that $\widetilde{\sigma}_{\mathcal{X}}\in \Upsilon(\Co X,\sigma)$ and that $\exp_{\mathcal{X}}$ is a
diffeomorphism between opens around $X$.
\end{lemma}
\begin{proof}
We identify the zero section of $T^*M$ with $M$, and for $x\in M$, we identify $T_{x}(T^*M)=T_xM\oplus T_x^*M$.
The properties of the Poisson spray imply that the geodesic flow fixes $M$, and that its differential along $M$ is given
by (see \cite{CrMar11})
\[d_{x}\phi_{\mathcal{X}}^t:T_xM\oplus T_x^*M\rmap T_xM\oplus T_x^*M, \ \ (Y,\xi)\mapsto (Y+t\pi^{\sharp}(\xi),\xi).\]

In particular, $\exp_{\mathcal{X}}=p\circ\phi_{\mathcal{X}}^1$ is a diffeomorphism around $X$, restricting to the identity along $X$, and the following formula for $\Omega_{\mathcal{X}}$ holds along $M$:
\[\Omega_{\mathcal{X}}\left((Y_1,\xi_1),(Y_2,\xi_2)\right)=\xi_2(Y_1)-\xi_1(Y_2)+\pi(\xi_1,\xi_2).\]
Taking $(Y_i,\xi_i) \in T_xX\oplus \Co_xX=T_x(\Co X)$, for $x\in X$, we obtain
\begin{align*}
\Omega_{\mathcal{X}}&\left((Y_1,\xi_1),(Y_2,\xi_2)\right)=\pi(\xi_1,\xi_2)=w_X(\xi_1,\xi_2),
\end{align*}
showing that $\widetilde{\sigma}_{\mathcal{X}}\in \Upsilon(\Co X,-w_X)$.
\end{proof}

Next, we observe that Theorem 3 implies the existence of self-dual pairs:

\begin{lemma}\label{Lemma_self_dual_pair}
Let $\mathcal{X}$ be a Poisson spray on the Poisson manifold $(M,\pi)$, and denote by $\Omega_{\mathcal{X}}$ the
symplectic form from Theorem \ref{existence of symplectic realizations}. On an open neighborhood of the zero section
$\Sigma\subset T^*M$ we have a full dual pair:
\[(M,\pi)\stackrel{p}{\lmap}(\Sigma,\Omega_{\mathcal{X}})\stackrel{\exp_{\mathcal{X}}}{\rmap}(M,-\pi).\]
\end{lemma}
\begin{proof}
Let $\Sigma$ be an open neighborhood of the zero section on which the geodesic flow $\phi_{\mathcal{X}}^t$ is
defined for all $t\in [0,1]$, and on which $\Omega_{\mathcal{X}}$ is nondegenerate. In the proof of the main result of \cite{CrMar11} (Theorem 3 above) it is shown that the symplectic orthogonals of the fibres $p$ are the fibres of $\exp_{\mathcal{X}}$. To show that $\exp_{\mathcal{X}}$ pushes $\Omega_{\mathcal{X}}^{-1}$ down to a bivector on $M$, one could
invoke Libermann's theorem, and then, using the formulas from the proof of Lemma \ref{lemma_well_defined}, one could check that along the zero section this bivector is indeed $-\pi$. We adopt a more direct approach. First note that $-\mathcal{X}$ is a Poisson spray for $-\pi$, and that on $\Sigma_{-}:=\phi_{\mathcal{X}}^1(\Sigma)$, the geodesic flow of $-\mathcal{X}$ is defined up to time 1. Moreover, $\Omega_{-\mathcal{X}}$ is nondegenerate on
$\Sigma_{-}$, because
\[(\phi_{\mathcal{X}}^1)^*\Omega_{-\mathcal{X}}=
\int_{0}^1(\phi_{\mathcal{X}}^1)^*(\phi_{-\mathcal{X}}^{t})^*\omega_{\textrm{can}}dt=
\int_{0}^1(\phi_{\mathcal{X}}^{1-t})^*\omega_{\textrm{can}}dt=
\int_{0}^1(\phi_{\mathcal{X}}^{t})^*\omega_{\textrm{can}}dt
=\Omega_{\mathcal{X}}.\]
This also finishes the proof, since $\exp_{\mathcal{X}}$ is the composition of Poisson maps:
\[(\Sigma,\Omega_{\mathcal{X}})\stackrel{\phi_{\mathcal{X}}^1}{\rmap}(\Sigma_{-},\Omega_{\mathcal{-X}})\stackrel{p}{\rmap}(M,-\pi).\qedhere \]
\end{proof}

We are ready to conclude the proof.

\begin{proof}[Proof of Theorem 1]
We use the self-dual pair from Lemma \ref{Lemma_self_dual_pair}, which, by abuse of notation, we write as if it were defined on the entire $T^*M$:
\[(M,\pi)\stackrel{p}{\lmap}(T^*M,\Omega_{\mathcal{X}})\stackrel{\exp_{\mathcal{X}}}{\rmap}(M,-\pi).\]
Using Lemma \ref{Lemma_restricting_dual_pairs}, we restrict to $X\times M$ to obtain a new dual pair (again, the
maps are defined only around $X$)
\[(X,\pi_{X})\stackrel{p}{\lmap}(T^*M|_{X},\Omega_{\mathcal{X}}|_{T^*M|_{X}})\stackrel{\exp_{\mathcal{X}}}{\rmap}(M,-\pi),\]
By Lemma \ref{lemma_Dirac}, we have the following equality of Dirac structures:
\[p^*(L_{\pi_{X}})^{-\Omega_{\mathcal{X}}|_{T^*M|_{X}}}=\exp_{\mathcal{X}}^*(L_{\pi}).\]
The left hand side restricts along $\Co X$ to the Dirac structure of the local model
$\pi(\widetilde{\sigma}_{\mathcal{X}})$, thus:
\[L_{\pi(\widetilde{\sigma}_{\mathcal{X}})}=\exp_{\mathcal{X}}^*(L_{\pi}).\]
Since $\exp_{\mathcal{X}}$ is a diffeomorphism around $X$ (Lemma \ref{lemma_well_defined}), we see that it is a
Poisson diffeomorphism around $X$:
\[\exp_{\mathcal{X}}:(\Co X,\pi(\widetilde{\sigma}_{\mathcal{X}}))\diffto (M,\pi).\qedhere\]
\end{proof}

\section{Application: Equivariant Weinstein splitting theorem}\label{Applications of the normal form theorem}

As an application of the normal form theorem (or rather of its proof), we obtain an equivariant version of Weinstein's
splitting theorem around fixed points. A version of this result with extra assumptions was obtained in
\cite{MirZung06}.

\begin{theorem2}
Let $(M,\pi)$ be a Poisson manifold and $G$ a compact Lie group acting by Poisson diffeomorphisms on $M$. If
$x\in M$ is a fixed point of $G$, then there are coordinates
$(p_1,\ldots,p_n,q_1,\ldots,q_n,y_1,\ldots,y_m)\in\mathbb{R}^{2n+m}$ centered at $x$ such that
\[\pi=\sum_{i=1}^n\frac{\partial}{\partial q_i}\wedge\frac{\partial}{\partial p_i}+\frac{1}{2}\sum_{j,k=1}^m \varpi_{j,k}(y)\frac{\partial}{\partial y_j}\wedge\frac{\partial}{\partial y_k},\ \ \varpi_{j,k}(0)=0,\]
and in these coordinates $G$ acts linearly and keeps the subspaces $\mathbb{R}^{2n}\times \{0\}$ and $\{0\}\times
\mathbb{R}^m$ invariant.

In other words, $(M,\pi)$ is $G$-equivariantly Poisson diffeomorphic around $x$ to an open around $(0,x)$ in the
product
\begin{equation}\label{EQProd}
(T_xL,\omega_{x}^{-1})\times (X,\pi_{X}),
\end{equation}
where $(L,\omega)$ is the symplectic leaf through $x$, $X$ is a $G$-invariant Poisson transversal of complementary
dimension, and $G$ acts diagonally on (\ref{EQProd}).
\end{theorem2}

\subsubsection*{On equivariant symplectic trivializations}

In the proof of Theorem 2 we will use a lemma on equivariant trivializations of symplectic vector bundles, which we
present here. We start with a result about symplectic vector spaces:

\begin{lemma}
Let $(V,\omega_0)$ be a symplectic vector space. There exist an open neighborhood $\mathcal{U}(\omega_0)$ of
$\omega_0$ in $\bigwedge^2 V^*$, invariant under the group $\mathrm{Sp}(V,\omega_0)$ of linear symplectomorphisms of
$\omega_0$, and a smooth map
\[b:\mathcal{U}(\omega_0)\rmap \mathrm{Gl}(V), \ \ \omega\mapsto b_{\omega}\]
satisfying :
\[b_{\omega}^*(\omega_0)=\omega, \ \ b_{\omega_0}=Id,\ \ s^{-1}\circ b_{\omega}\circ s=b_{s^*(\omega)},\]
for all $\omega\in \mathcal{U}(\omega_0)$ and all $s\in \mathrm{Sp}(V,\omega_0)$.
\end{lemma}
\begin{proof}
On the open $\mathbb{O}:=\mathbb{C}\backslash[0,\8)$ consider the holomorphic square-root
\[\sqrt{(\cdot)}:\mathbb{O}\rmap \mathbb{C},\ \ \sqrt{e^{a+i\theta}}:=e^{a/2+i\theta/2}, \ a\in\mathbb{R}, \ \theta\in (-\pi,\pi).\]
Denote the set of linear isomorphisms of $V$ with eigenvalues in $\mathbb{O}$ by $\mathbb{O}(V)\subset \mathrm{Gl}(V)$. By
holomorphic functional calculus \cite{wiki}, there is an ``extension'' of the square-root to $\mathbb{O}(V)$, which
satisfies:
\[\left(\sqrt{x}\right)^2=x,\ \sqrt{x^{-1}}=\left(\sqrt{x}\right)^{-1},\ \sqrt{(y\circ x\circ y^{-1})}=y\circ\sqrt{x}\circ y^{-1},\ \sqrt{x^*}=\left(\sqrt{x}\right)^*,\]
for every $x\in \mathbb{O}(V)$ and every linear isomorphism $y:V\to W$.

Consider $\mathcal{U}(\omega_0):=\{\omega_0\circ x | x\in\mathbb{O}(V)\}$, and define the map
\[b:\mathcal{U}(\omega_0)\rmap \mathrm{Gl}(V), \ b_{\omega}:=\sqrt{\omega_0^{-1}\circ \omega}.\]
Note that via the identification $\bigwedge^2V^*\subset \mathrm{Hom}(V,V^*)$, the action of $\mathrm{Gl}(V)$ on $\bigwedge^2V^*$ becomes $y^*(\omega)=y^*\circ\omega\circ y$. Let $\omega=\omega_0\circ x\in \mathcal{U}(\omega_0)$, with $x\in
\mathbb{O}(V)$ and $s\in \mathrm{Sp}(V,\omega_0)$. The following shows that $\mathcal{U}(\omega_0)$ is
$\mathrm{Sp}(V,\omega_0)$-invariant:
\[s^*(\omega)=s^*\circ \omega_0\circ x\circ s=(s^*\circ \omega_0\circ s)\circ (s^{-1}\circ x\circ s)=\omega_0\circ s^{-1}\circ x\circ s\in \mathcal{U}(\omega_0).\]
For the next condition, note first that
\[b_{\omega}^*=(\sqrt{\omega_0^{-1}\circ \omega})^*=\sqrt{\omega\circ \omega_0^{-1}}=\omega_0\circ
b_{\omega}\circ \omega_0^{-1};\] therefore:
\[b_{\omega}^*(\omega_0)=b_{\omega}^*\circ\omega_0\circ b_{\omega}=\omega_0\circ
b_{\omega}^2=\omega.\] Finally, for $s\in \mathrm{Sp}(V,\omega_0)$, we have that
\begin{align*}
s^{-1}\circ b_{\omega}\circ s&=\sqrt{s^{-1}\circ \omega_0^{-1}\circ \omega\circ s}=\sqrt{s^{-1}\circ \omega_0^{-1}\circ (s^*)^{-1}\circ s^*\circ \omega\circ s}=\\
&=\sqrt{\left(s^*(\omega_0)\right)^{-1}\circ s^*(\omega)}=b_{s^*(\omega)}.\qedhere
\end{align*}
\end{proof}
\begin{remark}\rm
The lemma can also be proved using the Moser argument. First note that $\mathcal{U}(\omega_0)$ can be
described as the set of 2-forms $\omega\in \bigwedge^2 V^*$ for which $\omega_t:=t\omega_0+(1-t)\omega$ is
nondegenerate for all $t\in[0,1]$. The 2-form $\omega-\omega_0$ has a canonical primitive given by $\eta:=\frac{1}{2}
\iota_{\xi}(\omega-\omega_0)$, where $\xi$ is the Euler vector field of $V$. Let $X_t(\omega)$ be the time-dependent
vector field defined by the equation $\iota_{X_t(\omega)}\omega_t=\eta$. The Moser argument shows that the time $t$
flow of $X_t(\omega)$ pulls $t\omega_0+(1-t)\omega$ to $\omega$, and one can easily check that $b_{\omega}$ is the
time-one flow of $X_t(\omega)$.
\end{remark}

\begin{lemma}\label{lemma_simlt_triv}
Let $(E,\sigma)\to X$ be a symplectic vector bundle, and let $G$ be a compact group acting on $E$ by symplectic
vector bundle automorphisms. If $x\in X$ is a fixed point, there exist an invariant open $U\subset X$ around $x$ and a $G$-equivariant symplectic vector bundle isomorphism
\[(E,\sigma|_{U})\diffto (E_x\times U, \sigma_{x}),\]
where the action of $G$ on $E_x\times U$ is the product one.
\end{lemma}

\begin{proof}
We first construct a $G$-equivariant product decomposition. Let $U$ be a $G$-invariant open over which $E$
trivializes, and fix a trivialization $E|_{U}\cong E_x\times U$. The action of $G$ on $E_x\times U$ is of the form
$g(e,y)=(\rho_y(g)e,gy)$. To make the action diagonal, we apply the vector bundle isomorphism
\[\alpha:E_x\times U\diffto E_x\times U, \ \ (e,y)\mapsto (A_y(e),y),\ \ A_y:=\int_{G}\rho_{x}(g)^{-1}\rho_{y}(g)d\mu(g),\]
where $\mu$ is the Haar measure on $G$. Note that $A_y$ is a linear isomorphism for $y$ near $x$, and that it satisfies
\[A_{gy}\circ \rho_{y}(g)=\rho_x(g)\circ A_y.\]
Thus, by shrinking $U$, we may assume that the action on $E_x\times U$ is the product action, which we simply
denote by $g(e,y)=(ge,gy)$.

The symplectic structures are given by a smooth family $\{\sigma_y\}_{y\in U}$ of bilinear forms on $E_x$. This family
is $G$-invariant, in the sense that it satisfies:
\[\sigma_{gy}=(g^{-1})^*(\sigma_{y}), \ \ g\in G, \ y\in U.\]
Consider the open $\mathcal{U}(\sigma_x)\subset \bigwedge^2 E_x^*$ and the map $b:\mathcal{U}(\sigma_x)\to
\mathrm{Gl}(E_x)$ from the previous lemma. By shrinking $U$, we may assume that $\sigma_y\in \mathcal{U}(\sigma_x)$, for all
$y\in U$. Since $b_{\sigma_y}^*(\sigma_x)=\sigma_y$, we have a ``canonical'' symplectic trivialization:
\[\beta: E_x\times U\diffto  E_x\times U, \ \ (e,y)\mapsto (b_{\sigma_{y}}e,y),\]
Now $g^{-1}:E_x\to E_x$ preserves $\sigma_x$, so:
\[b_{\sigma_{gy}}=b_{(g^{-1})^*\sigma_y}=g\circ b_{\sigma_y}\circ g^{-1}.\]
Equivalently, the map $\beta$ is $G$-equivariant:
\[\beta(ge,gy)=(b_{\sigma_{gy}}ge,gy)=(gb_{\sigma_y}e,gy)=g\beta(e,y).\]
Thus, $\beta\circ\alpha$ is an isomorphism of symplectic vector bundles that trivializes the symplectic
structure, and turns the $G$-action into the product one.
\end{proof}

\subsubsection*{Proof of Theorem 2}
We split the proof into 4 steps.

\underline{Step 1: a $G$-invariant transversal}. Let $(L,\omega)$ denote the leaf through $x$. Since $x$ is a fixed
point, it follows that $G$ preserves $L$. Thus $G$ acts by symplectomorphisms on $(L,\omega)$.

We fix $X\subset M$ a $G$-invariant transversal through $x$ such that $\dim(L)+\dim(X)=\dim(M)$. The existence of
such a transversal follows from Bochner's linearization theorem: the action around $x$ is isomorphic to the linear
action of $G$ on $T_xM$; by choosing a $G$-invariant inner product on $T_xM$, we let $X$ be an invariant ball around
the origin in the orthogonal complement of $T_xL$.

Let $\pi\vert_X=\pi_X+w_X$ denote the decomposition of $\pi$ along $X$. Then $G$ acts by Poisson diffeomorphisms
on $(X,\pi_{X})$, and by symplectic vector bundle automorphisms on $(\Co X,-w_X)$.

\underline{Step 2: the $G$-invariant spray}. Let $\mathcal{X}$ be a $G$-invariant Poisson spray. Such a vector field
can be constructed by averaging any Poisson spray; the conditions that a vector field on $T^*M$ be a Poisson spray
are affine. The flow of $\mathcal{X}$ is therefore $G$-equivariant. By the detailed version of Theorem 1, and with the
notations used there, we obtain a $G$-equivariant Poisson diffeomorphism around $X$
\[\exp_{\mathcal{X}}:(\Co X,\pi(\widetilde{\sigma}_{\mathcal{X}}))\rmap (M,\pi),\]
where $\widetilde{\sigma}_{\mathcal{X}}\in \Upsilon(\Co X,-w_X)$ is automatically $G$-invariant.

\underline{Step 3: a $G$-equivariant symplectic trivialization}. Note first that $w_X$, regarded as a map $\Co X \to
TM\vert_X$, yields a symplectic isomorphism\[w_{X,x}:(\Co_xX,-w_{X,x})\diffto (T_xL,\omega_x).\]This remark and
Lemma \ref{lemma_simlt_triv} imply that around the fixed point $x$, by shrinking $X$ if necessarily, we can
simultaneously trivialize the bundle $(\Co X,-w_X)$ symplectically and turn the action to a product action, hence, we
obtain a $G$-equivariant symplectic vector bundle isomorphism
\[\Psi:\left(\pr_2:(T_xL,\omega_x)\times X\to X\right)\diffto \left(p:(\Co X,-w_X)\to X\right),\]
where the action on $T_xL\times X$ is the product action. Therefore,
$\widetilde{\omega}_{\mathcal{X}}:=\Psi^*(\widetilde{\sigma}_{\mathcal{X}})$ is a closed $G$-invariant extension
of $\omega_x$, i.e.\ $\widetilde{\omega}_{\mathcal{X}}\in \Upsilon(T_xL\times X,\omega_x)$. Moreover, the map
\[\Psi:(T_xL\times X, \pi(\widetilde{\omega}_{\mathcal{X}}))\diffto (\Co X,\pi(\widetilde{\sigma}_{\mathcal{X}}))\]
is a $G$-equivariant Poisson diffeomorphism, where $\pi(\widetilde{\omega}_{\mathcal{X}})$ denotes the Poisson
structure around $X$ corresponding to the Dirac structure $\pr_2^*(L_{\pi_{X}})^{\widetilde{\sigma}_{\mathcal{X}}}$.

\underline{Step 4: the $G$-equivariant Moser argument}. Note that $\omega_x$ has a second extension to $T_xL\times
X$ given by $\overline{\omega}_x:=\pr_1^*(\omega_x)$. The corresponding local model is the Poisson structure from
the statement:
\[(T_xL\times X,\pi(\overline{\omega}_x))=(T_xL,\omega_x^{-1})\times (X,\pi_{X}).\]
By Steps 2 and 3, we are left to find a $G$-equivariant diffeomorphism around $X$ that sends
$\pi(\widetilde{\omega}_{\mathcal{X}})$ to $\pi(\overline{\omega}_x)$. For this we need the equivariant version of
Lemma \ref{lemma_indep}, whose proof can be easily adapted to this setting: First, note that the two-form
$\overline{\omega}_x-\widetilde{\omega}_{\mathcal{X}}$ has a primitive
\[\eta\in \Omega^1(T_xL\times X) \textrm{  such that  }\eta_{(0,y)}=0\textrm{ for all } y\in X.\]
Since both $\overline{\omega}_x$ and $\widetilde{\omega}_{\mathcal{X}}$ are $G$-invariant, by averaging, we can
make $\eta$ $G$-invariant as well. Consider the time-dependent vector field
\[Y_t:=-\pi_t^{\sharp}(\eta),\ \ \ \ \textrm{where }\pi_t:=\pi(\widetilde{\omega}_{\mathcal{X}})^{td\eta}.\]
The time-one flow $\phi_Y^{1,0}$ sends $\pi_0=\pi(\widetilde{\omega}_{\mathcal{X}})$ to
$\pi_1=\pi(\overline{\omega}_x)$. Since both $\pi_t$ and $\eta$ are $G$-invariant, it follows that $\phi_Y^{1,0}$ is
$G$-equivariant as well. This concludes the proof.


\begin{thebibliography}{9}
\bibitem{RuiOli} O.~Brahic, R.~L. ~Fernandes, \emph{Poisson fibrations and fibered symplectic groupoids}, \emph{Contem.~Math.}, AMS \textbf{450} (2008), 41--59.
\bibitem{BR} H.~Bursztyn, O.~Radko, \emph{Gauge equivalence of Dirac structures and symplectic groupoids}, Ann.~Inst.~Fourier (Grenoble)  \textbf{53}  (2003), no. 1, 309--337.
\bibitem{CattZamb09} A.~S.~Cattaneo, M.~Zambon, \emph{Coisotropic embeddings in Poisson manifolds}, Trans. Amer. Math. Soc. \textbf{361} (2009), no. 7, 3721--3746.
\bibitem{CrFer04} M.~Crainic, R.L.~Fernandes, \emph{Integrability of Poisson brackets}, J.~Diff.~Geom.~\textbf{66} (2004), 71--137.
\bibitem{CrMar11} M.~Crainic, I.~M\u{a}rcu\cb{t}, \emph{On the existence of symplectic realizations}, J.~Symplectic~Geom.~ \textbf{9}, (2011), no. 4, 435--444.
\bibitem{Dirac50} P.~Dirac, \emph{Generalized Hamiltonian dynamics}, Canad.~J.~Math. \textbf{2}, (1950), 129--148.
\bibitem{DZ} J.-P.~Dufour, N.T.~Zung, \emph{Poisson structures and their normal forms}, Progress in Mathematics, 242, Birkhauser Verlag, Basel, 2005.
\bibitem{Gromov86} M.~Gromov, \emph{Partial differential relations}, Ergebnisse der Mathematik und ihrer Grenzgebiete Vol. 9, Springer-Verlag, 1986.\textbf{58} (1983), 617-621.
\bibitem{Haef58} A.~H\ae fliger, \emph{Structures feuillet\'ees et cohomologie \`a valeur dans un faisceau de groupo\"{i}des}, Commentarii mathematici Helvetici, \textbf{32} (1957-58), 248--329.
\bibitem{MirZung06} E.~Miranda, N.~T.~Zung, \emph{A note on equivariant normal forms of Poisson structures}, Math.~Res.~Lett.~ \textbf{13}, (2006), no. 5-6, 1001--1012.
\bibitem{Vor} Y.~Vorobjev, \emph{Coupling tensors and Poisson geometry near a single symplectic leaf}, Banach
Center Publ. \textbf{54} (2001), 249--274.
\bibitem{Wein71} A.~Weinstein, \emph{Symplectic manifolds and their lagrangian submanifolds}, Adv.~Math.~ \textbf{6} (1971), 329--346.
\bibitem{Wein77} A.~Weinstein, \emph{Lectures on symplectic manifolds}, Regional Conference Series in Mathematics, No. 29. American Mathematical Society, Providence, R.I., 1977.
\bibitem{Wein83} A.~Weinstein, \emph{The local structure of Poisson manifolds},  J.~Diff.~ Geom.~\textbf{18} (1983), 523--557.%, Errata and addenda, \emph{J.~Differential Geom.~}\textbf{22} (1985), 255.
\bibitem{Xu03} P.~Xu, \emph{Dirac submanifolds and Poisson involutions}, Ann.~Sci.~\'{E}cole~Norm.~Sup. (4) \textbf{36} (2003), no. 3, 403--430.
\bibitem{Zamb11} M.~Zambon, \emph{Submanifolds in Poisson geometry: a survey}, Complex~and~differential~geometry, 403--420, Springer Proc. Math., 8, Springer, Heidelberg, 2011.
\bibitem{wiki} Wikipedia article \emph{Holomorphic functional calculus}, version: \url{http://en.wikipedia.org/w/index.php?title=Holomorphic_functional_calculus&oldid=544193271}.
\end{thebibliography}
\end{document}